\documentclass{article}
\usepackage{amsfonts}
\usepackage{tikz-cd}
\usepackage[utf8]{inputenc}
\usepackage{color}

\usepackage{mathrsfs}
\usepackage{url}
\usepackage{amsmath}
\usepackage{array}
\usepackage{graphicx}
\usepackage{amssymb}
\usepackage{amstext}
\usepackage{amsthm}
\usepackage{hyperref}
\usepackage{colonequals}
\usepackage{enumitem}
\usepackage[alphabetic,lite]{amsrefs}
\usepackage{cleveref}
\usepackage[all,cmtip]{xy}
\usepackage{fullpage}
\numberwithin{equation}{subsection}

\newtheorem{theorem}{Theorem}[section]
\newtheorem{lemma}[theorem]{Lemma}
\newtheorem{proposition}[theorem]{Proposition}

\newtheorem{corollary}[theorem]{Corollary}

\theoremstyle{definition}

\newtheorem{definition}[theorem]{Definition}
\newtheorem{conditions}[theorem]{Conditions}

\theoremstyle{remark}
\newtheorem{remark}[theorem]{Remark}

\newtheorem{rmk}[theorem]{Remark}

\newcommand{\C}{\mathbb{C}}

\DeclareMathOperator{\Loc}{Loc}
\DeclareMathOperator{\GL}{GL}

\DeclareMathOperator{\End}{\mathsf{End}}

\newcommand{\ra}{\rightarrow}

\newcommand{\A}{\mathbb{A}}

\newcommand{\D}{\mathbb{D}}
\newcommand{\E}{\mathbb{E}}
\newcommand{\F}{\mathbb{F}}

\newcommand{\M}{\mathbb{M}}
\newcommand{\N}{\mathbb{N}}

\newcommand{\Q}{\mathbb{Q}}
\newcommand{\ol}{\overline}
\newcommand{\ul}{\underline}
\newcommand{\Qbar}{\overline{\Q}}

\newcommand{\V}{\mathbb{V}}
\newcommand{\Z}{\mathbb{Z}}

\newcommand{\cA}{\mathcal{A}}

\newcommand{\cC}{\mathcal{C}}
\newcommand{\cD}{\mathcal{D}}
\newcommand{\cE}{\mathcal{E}}
\newcommand{\cF}{\mathcal{F}}

\newcommand{\cL}{\mathcal{L}}

\newcommand{\cO}{\mathcal{O}}

\newcommand{\cS}{\mathscr{S}}

\newcommand{\cU}{\mathcal{U}}
\newcommand{\cV}{\mathbb{V}}

\newcommand{\cX}{\mathcal{X}}

\newcommand{\fp}{\mathfrak{p}}

\newcommand{\ShimK}{S_K(G,X)}

\newcommand{\integralShimK}{\mathscr{S}_K(G,X)}

\renewcommand{\cS}{\mathscr{S}}

\newcommand{\Gm}{\mathbb{G}_m}
\newcommand{\Fpbar}{\ol{k}}

\newcommand{\dR}{_{\mathrm{dR}}}
\newcommand{\et}{{\mathrm{et}}}
\newcommand{\cris}{_{\mathrm{cris}}}

\newcommand{\FL}{_{\mathrm{FL}}}

\newcommand{\Fil}{{\mathrm{Fil}}}

\newcommand{\ord}{{\mathrm{ord}}}

\DeclareMathOperator{\tor}{tor}

\newcommand{\frob}{\varphi}

\DeclareMathOperator{\dr}{dR}

\DeclareMathOperator{\GSp}{GSp}
\DeclareMathOperator{\Sp}{Sp}

\DeclareMathOperator{\gr}{Gr}

\newcommand{\sIsom}{\mathcal{I}\hspace{-.09em}som\,}

\DeclareMathOperator{\Rep}{Rep}

\DeclareMathOperator{\Gal}{Gal}

\DeclareMathOperator{\Aut}{Aut}
\DeclareMathOperator{\Lie}{Lie}

\DeclareMathOperator{\der}{der}

\DeclareMathOperator{\Gr}{Gr}
\DeclareMathOperator{\bb}{BB}

\def\scrA{\mathscr{A}}

\DeclareMathOperator{\id}{id}

\title{Finiteness of function field-valued points on exceptional Shimura varieties}
\author{Benjamin Bakker, Ananth N. Shankar, Jacob Tsimerman}
\date{\today}

\begin{document}
\maketitle
\begin{abstract}
Let $C/k$ be a smooth curve over a finite field of characteristic $p>0$. We prove that there are finitely many principally polarized abelian schemes of given dimension $g$ over $C$ up to $p$-power isogeny. For curves over $\ol{k}$, we prove that the moduli space of such abelian schemes is finite type up to $p$-power isogeny. 

Moreover, we generalize this result to arbitrary (not necessarily abelian type) Shimura varieties $S$ and sufficiently large primes $p$ in terms of $S$: the space of generically ordinary  morphisms $C\to S_{k}$ (resp. $C\to S_{\ol{k}})$ is finite (resp. finite type) up to $p$-Hecke orbits.
\end{abstract}
\section{Introduction}

\subsection{Abelian varieties}

The starting point for our work is the classical Shafarevich conjecture for abelian
varieties, proven by Faltings \cite{faltings1983}: \textit{given a number field $K$ and a finite set of places $S$, there are only finitely many isomorphism classes of
abelian varieties of dimension $g$ over K with good reduction outside $S$. }

For function fields of curves in characteristic $p$, the naive analogue cannot hold because of the Frobenius morphism. Let $C/k$ be a smooth curve over a finite field, and $\cA/C$ a principally polarized abelian scheme. One may the pull back under the relative Frobenius map $F_{C/k}$ to obtain additional such schemes $(F_{C/k}^n)^*\cA/C$. These are all distinct so long as $\cA$ is not isotrivial, as their Hodge degree grows to infinity.

Moreover, simply accounting for Frobenius is insufficient, as can be seen, for instance, by looking at products $\cA_1\times (F_{C/k}^n)^*\cA_2$. We prove the following analogue:

\begin{theorem}\label{thm:mainavfinitefield}
    Let $k$ be a finite field of characteristic $p>0$, and $C/k$ a smooth irreducible curve. 
    For $g>0$, there are only finitely many abelian schemes $\cA\ra C$ which are principally polarized of (relative) dimension $g$, up to $p$-power polarized isogeny. The same statement holds for unpolarized $g$-dimensional abelian schemes\footnote{Recall that an abelian scheme $\cA/ C$ is always polarizable.} up to $p$-power isogeny.
\end{theorem}

If instead we work with the algebraic closure $\ol{k}$ the statement becomes false,  as can be seen by simply looking at constant abelian schemes. As an attempt at a suitable analogue in this setting, one may hope that, up to $p$-power polarized isogeny, all such abelian schemes occur in a finite type family. 

To formulate our result, recall that any abelian scheme gives a Hodge bundle $\omega_{\cA/C}$ on $\ol{C}$ which is pulled back from an ample bundle on the Baily-Borel compactification $\scrA^{BB}_g$. We refer to the degree of the pull-back of this bundle as the \textit{Hodge degree.}

\begin{theorem}\label{thm:mainavalgclosedfield}
Let $k$ be a finite field of characteristic $p>0$, $C$ a smooth irreducible curve over either $k$ or $\bar k$, and $g>0$. There exists an integer $M=M(g,C)$, such that any principally polarized abelian scheme over $C$ of dimension $g$ admits a $p$-power polarized isogeny to another such scheme of Hodge degree bounded by $M$.

Moreover, if we restrict to $\cA$ with maximal algebraic monodromy, then the set of $p$-power polarized isogeny classes is finite.
\end{theorem}

\begin{rmk}
This may be seen as an analogue to Deligne's result \cite{deligne1987} implying that the space of maps from a complex curve to $\scrA_g$ is of finite type.
\end{rmk}

\begin{rmk}
Over $\ol{k}$ it is unclear to the authors whether the analogue of \Cref{thm:mainavalgclosedfield} is true.
\end{rmk}

\subsection{General Shimura varieties}

We now state our generalization to arbitrary Shimura varieties. To even make a well-defined statement we require the existence of good integral models, which exist for sufficiently large primes by \cite{BST}. The analogue of $p$-power isogeny is provided by Hecke operators at $p$, for which we simply take the Zariski closures of the characterstic zero operators to define their characteristic $p$ analogues (see \Cref{def:pheckeorbit}).

\begin{theorem}\label{thm:mainshimuravarieties}
Let $S$ be a Shimura variety with reflex field $E$, and $\cS/\cO_E[1/N]$ be the integral canonical model in the sense of \cite{BST} satisfying \Cref{cond:integralmodel}. Let $k$ be a finite field, and $C$ a smooth curve over either $k$ or $\ol{k}$. There is an integer $M=M(C,S)$ such that for any morphism $f:C\to S$ whose image intersects the ordinary locus, there exists a morphism $f_1:C\to S$ of Griffiths degree at most $M$ such that $f,f_1$ are in the same $p$-Hecke orbit (\Cref{def:pheckeorbit}). Hence, if $C$ is defined over $k$, the set of such $p$-Hecke orbits of morphisms is finite.

Finally, the set of $p$-Hecke orbits of such $f$ with maximal algebraic monodromy is finite. 
\end{theorem}

\subsection{Proof outline}

One of the main difficulties of deducing the statement from the Tate conjecture is dealing with infinitely many possible primes that may show up as the degree of an isogeny, each of which may potentially occur once. We observe that in the characteristic $p$ case, it is sufficient to control the relevant data at $p$ only (the $p$-divisible group in the abelian variety setting, and the $F$-crystal in general). This is simply because these data recover the Griffiths bundle which is ample, and so controlling its degree gives the finite typeness we require.

There are now two main steps to the proof:
\begin{enumerate}
    \item Proving finitely many $F$-isocrystals occur,
    \item Showing that one may move between $F$-crystals occurring in the same $F$-isocrystal using a $p$-Hecke operator.
\end{enumerate}

We accomplish (1) using a theorem of Litt (See \Cref{thm:Fisocrystalsfinitestuff}) and the theory of companions. For (2), in the context of abelian varieties it is straightforward since (polarized) isogenies of $p$-divisible groups yield (polarized) isogenies of the corresponding abelian varieties. In the context of Shimura varieties, this constitutes a significant challenge. Indeed, our Hecke operators are defined rather coarsely by taking a Zariski closure from characteristic 0, and they lack a motivic interpretation.

It is only in this part of the argument that we rely on the ordinary locus. Indeed, by the crystallinity results of Esnault-Groechenig \cites{EG,EG1} we obtain the crystallinity of $_{\et}\V_p$ on the generic fiber. As such, by the existence of slope filtrations, over the ordinary locus we are able to show that the (twisted) graded components of these local systems extend. Hence, we are able to preserve some part of the $G$-structure (the Levi part) even in characteristic $p$. We are then able to give interpretations for certain Hecke operators -  in terms of level structures corresponding to framing data on these graded bundles - even in characteristic $p$. This provides enough of a replacement for isogenies of abelian varieties for the proof to go through.

\subsection{Outline}
In \S2 we define $G$-structures on \'etale local systems, in a form convenient for us, and prove finiteness results for $F$-isocrystals using the work of Litt\cite{Litt2}. In \S3 we prove our result for abelian varieties. In \S4 we recall background around Shimura varieties, focusing in particular on level structures and Hecke operators in a way that will extend to mixed characteristic. In \S5 we recall background on $F$-crystals. In \S6 we study $F$-crystals over the ordinary locus, and use this to extend our Hecke morphisms to mixed chracteristic over (certain components of) the ordinary locus. Finally, in \S7 we prove our main result for Shimura varieties, following the outline of \S3 using the machinery developed in \S4-6.
\subsection{Acknowledgments}
We thank Anna Cadoret, Daniel Litt, Keerthi Madapusi and Si-Ying Li for useful conversations.

B. B. was partially supported by NSF grant DMS-2401383, the Institute for Advanced Study, and the Charles Simonyi Endowment. A.S. was partially supported by the NSF grant DMS-2338942, the Institute for advanced studies (via the NSF grant DMS-2424441), and a Sloan research fellowship. J.T. was partially supported by the Institute for Advanced Study, and the Charles Simonyi Endowment.
\section{Local Systems and $G$-structures}
\begin{definition}\label{def:G-structures}
Let $\Lambda$ be a topological group, and $\cF$ an pro-\'etale sheaf of sets/modules/vector spaces over a scheme $X/k$ over a field, in the sense of \cite[\S3]{scholze2013}. We define a $\Lambda$-structure on $\cF$ to consist of a pair $(S,\gamma)$, where $S$ is a $\Lambda$-set/module/vector space, and $\gamma\in\Gamma\big(X,\Lambda\backslash\sIsom(\cF,\ul{S})\big) $ where the $\Lambda$ action is trivial on $\cF$ and $\Lambda\backslash\sIsom(\cF,\ul{S})$ is the quotient in the category of sheaves (namely, the sheafification of the natural quotient presheaf). We also call this a $(\Lambda,S)$-structure if we wish to specify the $S$. We call $\gamma$ the \textit{trivializing section}.

A \textit{homomorphism} of sheaves with $\Lambda_1,\Lambda_2$-structures consists of compatible homomorphism of groups, sheaves, sets, and trivializing sections.

\end{definition}

Now if $G/\Q_p$ is an algebraic group, we define  a \textit{$\Q_p$-\'etale local system with $G$-structure}, as simply a morphism of Tannakian categories $\Rep_{\Q_p} G\to \Loc_{\Q_p}(X)$.

\begin{definition}
We say that a $\ol{\Q}_\ell$-local system on $C_k$ is \emph{algebraic} if the Frobenius eigenvalues at every closed point $x\in C(\F_q)$ are algebraic numbers in $\overline{\Q}$. We similarly define \emph{algebraic} $F$-isocrystals with $\ol{\Q}_p$ on $C_k$, with the semi-linear Frobenius at $x\in C(\F_q)$ being replaced by the eigenvalues of the $q$-Frobenius\footnote{Note that this is a linear endomorphism of the isocrystal restricted to $x$.} at $x$. 

Following \cite{Litt2}, we define an \textit{arithmetic} $F$-(iso)crystal on $C_{\bar k}$ to be an $F$-(iso)crystal which descends to $C_{k'}$ for some finite field $k'/k$. We likewise define arithmetic overconvergent $F$-(iso)crystals, and 
arithmetic overconvergent $F$-(iso)crystals with $L$ coefficients for $L\subset \ol{\Q}_p$. We note that every arithmetic semisimple $\ell$-adic local system (resp. overconvergent $F$-isocrystal) admits a descent to a mixed semisimple $\ell$-adic local system (resp. mixed semisimple $F$-isocrystal) over $C_{k'}$.

 \end{definition}

We will use the term companion for semisimple $F$-isocrystals and $\ell$-adic local systems in the sense of \cite[Definition 2.7]{kedlaya2022etale}, i.e. without requiring the objects to be irreducible or requiring them to have finite determinant. Further, note that by \cite[Corollary 3.15]{kedlaya2022etale}, given companion objects, the $F$-isocrystal is tame if and only if the $\ell$-adic local system is tame.  


 \begin{definition}

We call a semisimple $F$-isocrystal on $C_k$ with $\ol{\Q}_p$-coefficients 
\textit{$\Q_\ell$-nice} if its companion admits a descent to a semisimple $\Q_\ell$-local system. We call an arithmetic semisimple $F$-isocrystal on $C_{\ol{k}}$ with $\ol{\Q_p}$-coefficients \textit{$\Q_\ell$-nice} if it descends to a \textit{$\Q_\ell$-nice} semisimple $F$-isocrystal over $C_{k'}$.
Finally, as for local systems, if $G/\Q_p$ is an algebraic group, we define an \textit{overconvergent $F$-isocrystal with $G$-structure}, as simply a morphism of Tannakian categories $\Rep_{\Q_p} G\to F\mbox{-}\mathrm{Isoc}^{\dagger}(X)$.
 

\end{definition}

\begin{theorem}\label{thm:Fisocrystalsfinitestuff}
Let $C/k$ be a curve over a finite field with $q$ elements. Let $S$ be a finite set of rational numbers. 
\begin{enumerate}
   \item Any $F$-isocrystal with $\ol{\Q}_p$-coefficients on $C_{\ol{k}}$ (or $C_k$) admits finitely many descents to $\Q_p$ up to isomorphism. The same statement holds for $\ell$-adic local systems.
    
    \item  Fix a rank $r$, a point $x\in C(k')$, $k'/k$ a finite extension, and a finite set $T\subset \Qbar$. Up to isomorphism, there are only finitely many semisimple $F$-isocrystals of rank $r$ such that $\frob_x$ has eigenvalues lying in $T$.

    \item Fix a rank $r$ and a point $x\in C(\bar k)$. There are only finitely semisimple, tame, overconvergent, arithmetic, $\Q_\ell$-nice $F$-isocrystals of rank $r$ on $C_{\ol{k}}$ whose slopes at $x$ are contained within $S$. 

    \item For a reductive algebraic group $G\subset \GL_n$, any semisimple $F$-isocrystal over $C_{\bar{k}}$ (or $C_k$) admits finitely many $G$-structures up to isomorphism.

\end{enumerate}
    
\end{theorem}

\begin{proof}
\textbf{(1):} Let $E$ be such an $F$-isocrystal with $\ol{\Q}_p$-coefficients. Then the set of descents is given $H^1(\Gal_{{\Q}_p},\Aut(E))$. Note that $\Aut(E)$ is a linear algebraic group, and so this cohomology group is finite since first Galois cohomology groups valued in a linear algebraic group are finite \cite{SerreGalois}. An identical argument goes through in the $\ell$-adic case.

    \medskip

    \textbf{(2): } We may and do restrict to irreducible such isocrystals. Any semisimple, tame, overconvergent, mixed, $\Q_\ell$-nice $F$-isocrystal on $C_k$ is the companion of a semisimple, tame, mixed $\Q_\ell$-local system on $C_k$. It suffices to prove finiteness after extending coefficients to $\ol{\Q}_\ell$ by the first part. By a theorem of Deligne (\cite[Theorem 1.1]{esnault2012}), there are only finitely many irreducible $\ol{\Q}_\ell$-local systems on $C_k$ up to twist by a character of $\Gal(\ol{k}/k)$. It suffices to prove that there are only finitely many characters $\chi$ such that $V\otimes \chi$ satisfies the conditions on Frobenius eigenvalues. As the Frobenius eigenvalues at $x$ are constrained to a finite set, the possibilities for $\chi$ are finite, which completes the proof of (2). 
    
    \medskip

    \textbf{(3): } We can and will assume that the $F$-isocrystal does not have a non-trivial $\Q_\ell$-nice sub $F$-isocrystal. Any semisimple, tame, overconvergent, arithmetic, $\Q_\ell$-nice $F$-isocrystals on $C_{\ol{k}}$ is the companion of a semisimple, tame, $\Q_\ell$ local system on $C_{k'}$ for some finite subfield $k'$. By \cite[Cor 1.1.5(3)]{Litt2}, there are finitely may such local systems on $C_{\bar k}$.
    
    Fix such a local system $L$ on $C_{\bar k}$.
    Since $L$ has an irreducible descent to $C_{k'}$ it follows that $L_{\ol{\Q}_\ell}\cong M^m$ where $M$ is an irreducible $\ol{\Q}_\ell$-local system. Now since $L$ admits an extension to $C_{k'}$ so must $M$, call it $M_0$. Then since $M$ is irreducible, it follows that all extensions of $M$ to any $C_{k''}$ are of the form $M_0(\chi)$, where $\chi$ is a character of $\Gal_{k''}$. Thus, we see that our $F$-isocrystals are of the form $\bigoplus_i \cL\otimes \chi_i$ where $\chi_i$ is a 1-dimensional $F$-isocrystal pulled back from a point. As such $F$-isocrystals are completely determined by their slope, the claim follows from the finiteness of $S$.

    \medskip

    \textbf{(4):} We first prove this for $F$-isocrystals with $\ol{\Q}_p$-coefficients. Indeed, for such an $E$, this amounts to embedding the monodromy group $H$ of $E$ into $G$ up to conjugacy. To prove finiteness of such embeddings,  is enough to look at maps of Lie algebras since $H$ is semisimple.  The deformation space of such embeddings is classified by $H^1(\Lie H,\Lie G)$ which vanishes by the first Whitehead lemma, which implies the claim.

    Now it is enough to show that there are finitely many such descents. Indeed, this works exactly as above where we interpret $\Aut(E)$ as automorphisms respecting the $G$-structure.

\end{proof}

\section{Abelian Variety Statements}

We first deduce the unpolarized statement from the polarized statement:

\begin{proof}[\Cref{thm:mainavfinitefield} Polarized $\Rightarrow$ \Cref{thm:mainavfinitefield} Unpolarized]
Let $C/k$ be a curve as in the theorem.
The Tate conjecture - proven by Zarhin and Mori \cite{zarhin1975,zarhin1976,moret1985} in this setting - says that there are finitely many isogeny classes of abelian schemes over $C$. Hence, up to $p$-power isogeny, there are finitely many prime-to-$p$ power isogeny classes. Now taking prime-to-$p$ isogenies preserves the $p$-divisible group scheme, we may assume $\cA/C$ is such that $\cA[p^{\infty}]$ is isomorphic to a fixed $A$. We show there are finitely many such $\cA/C$.

We shall use Zarhin's trick, so set $Z:=(\cA\oplus \cA^{\vee})^4$. Then $Z$ is principally polarized and hence comes from a map $f:C\to \scrA_g$. Moreover, the Hodge degree of $f$ is fixed since the Hodge bundle on $C$ is recoverable as $\det\ker(F:\D((A\oplus A^\vee)^4)_{\ol{C}}\to \D((A\oplus A^\vee)^4)_{\ol{C}}^{(p)})$ where $\D((A\oplus A^\vee)^4)_{\ol{C}}$ denotes the corresponding log-flat bundle gotten by evaluating the contravariant log Dieudonn\'e module on ${\ol{C}}$. Since the Hodge bundle is ample there are finitely many such $f$ and hence finitely many such $Z$. Finally, since $A$ is a direct summand of $Z$, the statement follows from \cite[Thm 18.7]{milne1986av}.

\end{proof}

The remaining parts of \Cref{thm:mainavalgclosedfield} and \Cref{thm:mainavfinitefield} are implied by the following statement:

\begin{theorem}\label{thm:mainavcharp}
    Let $k$ be  a finite field of characteristic not dividing $N$. Let  $C$ be a smooth irreducible curve over either $k$ or $\bar{k}$. 
    For $g>0$, there is an integer $M=M(C,g)$ such that any principally polarized abelian schemes $\cA\ra C$ of dimension $g$ is $p$-power polarized isogenous to another such scheme $\cA'$ with Hodge degree bounded by $M$. Consequently, if $C$ is defined over $k$, there are finitely many such schemes up to isomorphism.

    Moreover, if $C/\ol{k}$ and we restrict to $\cA/C$ with algebraic monodromy $\Sp_{2g}$, the set of such $p$-power polarized isogeny classes is finite.
\end{theorem}

\begin{proof}

We first claim that there are finitely many principally polarized $p$-divisible groups occurring as $\cA[p^\infty]$ on $C_k$ (or $C_{\ol k}$) up to polarized isogeny. The functor from $\cA[p^\infty]$ to overconvergent $F$-isocrystals with $\GSp$ structure is fully faithful, and therefore the claim follows by \Cref{thm:Fisocrystalsfinitestuff}(4), and (2) (resp. (3)) for $C_k$ (resp. $C_{\ol k}$) once we meet the conditions. For (3), the $\Q_\ell$-niceness follows from the rational Tate module being defined over $\Q_\ell$. For (2), note that at any $x\in C(k')$ the Frobenius eigenvalues are Weil-$|k'|^{\frac 12}$ numbers by the Weil conjectures and hence constrained to a finite st.
 
Hence we may find a fixed $p$-divisible subgroup $B\supset \cA[p^{\infty}]$ of fixed isomorphism type, such that $\cA'=\cA/(B/\cA[p^{\infty}])$ is principally polarized with $p$-divisible group isomorphic to $B$. 
According to \cite[IV.5.7, V.2.3, VI.1.1]{faltingschai}, on a toroidal compactification of a sufficient level cover $\bar\scrA_{g,n}$ of $\scrA_g$ defined over $\Z_p$, the Hodge bundle $\omega$ may be identified with the top exterior power of $\Fil^1$ of the log Fontaine-Laffaile module associated to $\cU$, where $\cU$ is the universal semi-abelian scheme on $\bar\scrA_{g,n}$. Therefore, $\omega$ on $\bar\scrA_{g,n,\F_p}$ may be identified with $\det\ker(F:\D(\cU[p^\infty])\to \D(\cU[p^\infty])^{(p)})$ where $\D(\cU[p^\infty])$ denotes the log-flat bundle gotten by evaluating the contravariant log Dieudonn\'e module of $\cU_{\F_p}$ on $\bar\scrA_{g,n,\F_p}$.  The same is then true for the stack quotient, and so the Hodge bundle on $C$ is identified with $\det\ker(F:\D_{\bar C}\to \D^{(p)}_{\bar C})$, where $\D_{\bar C}$ is the unique log-$F$-crystal extension of $\D(B)$ guaranteed by \Cref{lem:crystaluniquelyextendstobdrypoints}.  The first part of the theorem now follows.


Finally, we prove the last part of the theorem. Suppose that $C$ is defined over $\ol{k}$ and  $\cA/C$ has algebraic monodromy $\Sp$. By the above, all such abelian schemes over $C$ arise from bounded degree maps to (the stack) $\scrA_g$. Since the moduli of such maps is of finite type, it is sufficient to prove that if $\cA'/C\times X$ is a principally polarized abelian scheme, with $X$ a smooth and connected variety over $\ol{k}$,  such that $\cA'_{C\times x_0}\cong \cA$, then all fibers $\cA'_{C\times x}$ admit a polarized $p$-isogeny to $\cA$. 

To see this, note that image of $\pi_1^{\et}(X)$ in $\ell$-adic monodromy commutes with that of $\pi_1^{\et}(C)$, and is therefore trivial. Now over some finite field $k'/k$ all of $C,\cA,\cA'$ descend. Applying the Tate conjecture over $k'$ we see that $\cA'$ is isogenous to $\pi_C^*\cA$, and thus the $p$-divisible group schemes are isogenous. Since the monodromy is maximal there are no non-trivial endomorphisms, and hence this isogeny must be a polarized isogeny. The claim follows.

\end{proof}

\section{Background: Shimura Varieties}
We collect some basic facts about Shimura varieties here. We refer the reader to \cite{MilneIntro} for background, and \cite{BST} for the specific setup used here.

Let $S=S_K(G,X)$ be a Shimura variety. We let $E$ denote the reflex field.
$\ShimK$ carries a flat principal $G$-bundle $\mathcal{E}/S$, which admits a canonical model over $E$. A choice of Hodge co-character defines a parabolic subgroup $P\subset G$, and gives a parabolic bundle $\mathcal{P}/S$, along with a canonical morphism $\mathcal{P} \rightarrow \mathcal{E}$, where all this data descends to $E$. We note that $\mathcal{P}$ is not flat for the connection on $\mathcal{E}$. 

Let $V$ denote a faithful rational representation of $G$, and $\V\subset V$ denote a lattice. The choice of $\V$ specifies a compact open subgroup of $ G(\A_f)$ by considering the stabilizers of $\V_{\hat\Z}$. We choose the level structure $K$ to be a subgroup of this compact open subgroup, such that it is neat, and such that it acts trivially on $\V/3\V$. Corresponding to $V$ there is a family of $\Q_\ell$ local systems $_{\et}V_\ell$ with $\Z_\ell$ sub-systems $_{\et}\V_\ell$. Moreover, there is a filtered flat vector bundle ${\dR}V$ over $\ShimK$ -- indeed we have that ${\dR}V$ is defined by $\mathcal{E} \times_G V$. The filtration is induced by the fact that ${\dR}V$ is also canonically isomorphic to $\mathcal{P}\times_{G_{E}} V_{E}$. Now, let $p$ be a prime and let $v \mid p$ be a prime of $E$. 

By choosing $N$ to be a large enough integer, we may assume that $\ShimK$ spreads out to a smooth integral model $\mathscr{S}=\integralShimK/\cO_{E}[1/N]$ such that the following conditions are satisfied:

\begin{conditions}\label{cond:integralmodel}\hspace{1em}

\begin{itemize}

    \item $G$ has a reductive model $\mathcal{G} / \Z[1/N]$. There are elements $s_{\alpha} \in \V[1/N]^{\otimes}$ whose point-wise stabilizers define the group $\mathcal{G}$.
    \item For every $p\nmid N$, we have $K =\mathcal{G}(\Z_p) \cdot K^p$, where $K^p$ is the prime-to-$p$ level structure, and $K$ is neat.
    \item $P\subset G_{E}$ is induced by a parabolic subgroup of $\mathcal{G}$ defined over $\cO_{E}[1/N]$.
    \item The map $\mathcal{P} \rightarrow \mathcal{E}$ is induced by a map $\mathscr{P} \rightarrow \mathscr{E}$ over $\mathscr{S}$, where $\mathscr{E}$ is a flat principal $\mathcal{G}$ bundle and $\mathscr{P}$ is a principal bundle for the parabolic of $\mathcal{G}$ defined over $\cO_{E}[1/N]$.
    \item $\integralShimK$ admits a log-smooth compactification $\ol{\integralShimK}$ over $\cO_{E}[1/N]$. 
    \item The Baily-Borel compactification $\ShimK^{\bb}$ has a compatification $\integralShimK^{\bb}$ and a natural map $\ol{\integralShimK}\to \integralShimK^{\bb}$.
    \item The local systems $_{\et}\V_{\ell}$ extend to $\integralShimK[1/\ell]$.
    \item For each finite place $v$ of $\cO_E[1/N]$, $_{\et}\V_p(m)$ is (log)-crystalline\footnote{By which we mean is associated to a log Fontaine-Laffaille module.} on $\integralShimK_{\cO_{E_v}}$ for $m$ sufficiently small so that $V(m)$ has positive Hodge weights. We define $_{\et}\V'_p:={_{\et}\V_p(m)}$.
    \item For each finite place $v$ of $\cO_E[1/N]$, let $_{\FL}\V'/\integralShimK_{\cO_{E_v}}$ denote the log-Fontaine-Laffaille module corresponding to $_{\et}\V'_p$ under the log-Fontaine-Laffaille correspondence. We let $\cris \V'$ be the underlying $F$-crystal, and $\cris V'$ the associated $F$-isocrystal.  Note that $\cris V'$ is semisimple, overconvergent, and tame (again using \cite[Cor 3.15]{kedlaya2022etale}).
    
    \item In general, by the Griffiths bundle of a vector bundle equipped with a locally split filtration $(M,\Fil^\bullet M)$ we mean $\otimes_i \det \Fil^iM$.  The Griffiths bundle of $_{\FL}\V'$ extends naturally to $\ol{\integralShimK}_{\cO_{E_v}}$, and descends amply to $\integralShimK^{\bb}_{\cO_{E_v}}$.
    \item The filtered flat vector bundle ${\dR}V':={{\dR}V(m)}$ (where the twist here denotes shifting the filtration) extends to an integral filtered flat vector bundle ${\dR}\cV'$, which is naturally identified with the filtered flat bundle underlying $_{\FL}\V'$.  The Kodaira-Spencer morphism $T_{S}\to \End(\Gr_{\Fil}{_{\dr}\cV})$ is immersive.
\end{itemize}

\end{conditions}

For the remainder of the paper, we fix a height 1 prime $v$ of $E$ such that $v\mid p, p\nmid N.$  We also drop ``$_K(G,X)$" from the notation and write $S=S_K(G,X)$, $\mathscr{S}=\mathscr{S}_K(G,X)$ for the integral canonical model, etc.  Finally, for $K'\subset K$ an arbitrary closed subgroup, we obtain a pro-variety $S_{K'}:=\displaystyle\varprojlim_{K'\subset K_1} S_{K_1}(G,X)$, where $K_1$ varies over compact open subgroups contained in $K$.

\renewcommand{\integralShimK}{\mathscr{S}}
\renewcommand{\ShimK}{S}
\subsection{Level structures}

By the assumption on $v$, the Hodge-co-character of $G_v$ has conjugacy class defined over $E_v\cong \Q_p$. Let $T_v\subset G_v$ be a maximal torus contained in a Borel. Then we have the following:

\begin{proposition}\label{prop: cocharacters recall}
    
    \begin{enumerate}
        \item Every $\Z_p$-valued co-character of $G$ is conjugate to a $\Z_p$-valued co-character valued in $T$.
        
        \item Let $\mu'$ denote a co-character of $G_{\overline{\Q}_p}$ whose conjugacy-class is defined over $\Q_p$. There is a co-character $\mu: \Gm \rightarrow T$ defined over $\Q_p$ whose base-change to $\overline{\Q}_p$ is conjugate to $\mu'$.

        \item Let $\mu_1,\mu_2$ be two $\Z_p$-valued co-characters that are conjugate over $\overline{\Q}_p$. Then, $\mu_1$ and $\mu_2$ are conjugate over $\Z_p$. 
    \end{enumerate}
\end{proposition}

\begin{proof}
The first part follows from \cite[Proposition 3.2.1 (iii)]{ChaoZhang}. The second part follows from \cite[Lemma 1.1.3]{Kottwitzcochar}. For the third part, we may assume that both $\mu_1$ and $\mu_2$ are valued in $T_S\subset T$, where $T_S$ is the maximal split subtorus of $T$. By \cite[Lemma 1.1.3 (a)]{Kottwitzcochar}, we have that $\mu_1$ and $\mu_2$ are conjugate by an element $w$ in the relative Weyl group of $T_S$, $N(T_{S,\Q_p})/Z(T_{S,\Q_p})$. \cite[Equation (1.1.3.1)]{Kottwitzcochar} implies that we may choose a $G(\Z_p)$-valued representative of $w$, as required. 
\end{proof}

By the proposition, we may take a Hodge co-character $\mu\in X(\Z_p)$. Let $P\subset G_v$ denote the corresponding negative parabolic subgroup. 
Let $U$ be the unipotent radical of $P$, $L$ the group $P/U$ which we identify with the Levi subgroup $L_0=Z(\mu)$ associated to $\mu$.
Note that the Shimura axioms imply that $U$ is commutative, and hence the conjugation action of $P$ on $U$ factors through $L$. 

\begin{definition}\label{def:positivePelement}
    Let $g\in L_0(\Q_p)$ be an element. We say that $g$ is semi-positive if $U(\Z_p)\subset gKg^{-1}$.
\end{definition}

We have the following elementary lemma. 
\begin{lemma}\label{lem:semipositivityproperty}
\begin{enumerate}
    \item $g\in L_0(\Q_p)$ is semi-positive iff $U(\Z_p)\subset gU(\Z_p)g^{-1}$
    \item For any $g\in L_0(\Q_p)$, the element $\mu(p^n)g$ is semi-positive for sufficiently large $n$.
    \item If $g\in L_0(\Q_p)$ is semi-positive, we have the equality $P(\Z_p)\cap gKg^{-1} = P(\Z_p)\cap \big(gKg^{-1}\cdot U(\Q_p)\big)$
\end{enumerate}
\end{lemma}
\begin{proof}
  The first statement follows from $g^{-1}U(\Z_p)g\subset U(\Q_p)$. The second statement follows from the fact that $\mu(p) U(\Z_p)\mu(p)^{-1} = \frac{1}{p}U(\Z_p)$, where we view $U(\Z_p)$ as a free $\Z_p$-submodule of the vector space $U(\Q_p$.
  For the third statement, note that $P(\Z_p)=L_0(\Z_p)U(\Z_p).$ So suppose that $t\in P(\Z_p)\cap \big(gKg^{-1}\cdot U(\Q_p)\big)$, and write $t=gkg^{-1}u$, where $k\in K, u\in U(\Q_p).$ Note that we must have $k\in P(\Z_p)$, so we write $k=l_0u_0$ for $l_0\in L_0(\Z_p),u_0\in U(\Z_p)$. It follows that $t=(gl_0g^{-1})\cdot (gu_0ug^{-1})$. Hence, we must have that $gl_0g^{-1}\in L(\Z_p)$. Thus we may write $t=(gl_0g^{-1})\cdot u_1$ where $u_1\in U(\Z_p)$, and by (i) it follows that $gu_2g^{-1}$ for $u_2\in U(\Z_p)$. Thus $t= g(l_0u_2)g^{-1}\in P(\Z_p)\cap gKg^{-1}$ and the claim follows.
    
\end{proof}

We define $S_P:=S_{P(\Z_p)K^p}$ and $\pi_P:S_P\to S$ the natural morphism.  Recall that $_{\et}\V_p$ has a $(G(\Z_p),\V_p)$-structure, in the sense of \Cref{def:G-structures}. 

\begin{lemma}\label{lem:lift}
Let $Y/E_v$ be a variety with a $E_v$-morphism $\phi:Y\to S$, and set $M:=\phi^*_{\et}\V_p$, so that $M$ has a $G(\Z_p)$-structure. 

\begin{enumerate}
    \item A lift $\tilde{\phi}:Y\to S_{P}$ is equivalent to a filtration $W_\bullet M$ such that $(M,W_\bullet M)$ has a $(P(\Z_p),(\V_p,W_\bullet\V_p))$-structure compatible with the $G(\Z_p)$-structure on $\V_p$.
    \item Let $g\in P(\Q_p)$ be a semi-positive element. A lift $\tilde{\phi}:Y\to S_{P(\Z_p)K^p\cap gKg^{-1}}$ is equivalent to the above data, together with a graded lattice $\oplus_i M'_i\subset \Gr_W M\otimes \Q_p$, for which $(\Gr_W M,\oplus_i M_i')$ has a $\big(L(\Z_p),(\Gr_W \V_p, \Gr_W g\V_p)\big)$-structure compatible with the $(P(\Z_p),(\V_p,W_\bullet\V_p))$-structure on $(M,W_\bullet M)$.
\end{enumerate} 

The analogous statements hold for the analytifications of the Shimura varieties in the rigid analytic category.
\end{lemma}

\begin{proof}

First note that the correspondence one way is easy, since the local systems on $S_{P},S_{P\cap gKg^{-1}}$ are naturally endowed with the given structures.

For the reverse direction, for part (i) consider the \'etale covering space $S^{P}\to S$ of filtrations of $_{\et}\V_p$ as given. Then the presence of such a filtration on $\pi_{P}^*(_{\et}\V_p)$ gives  a natural morphism $S_{P}\to S^{P}$  which  is easily seen to be an isomorphism over $\C$, and hence must be an isomorphism also over the reflex field. The result follows. 

For part (ii), first note that by arguing similarly we see that the data of a lift is equivalent to a filtration $W_\bullet M$ as described, together with another lattice $M'\subset M\otimes \Q_p$ such that $(M,W_{\bullet}M,M')$ has a $(P(\Z_p)\cap gK_pg^{-1},(\V_p,W_{\bullet}\V_p,g\V_p))$-structure compatible with the $(P(\Z_p),(\V_p,W_\bullet\V_p))$-structure on $M$.

Now, the possible lattices $M'$ correspond, under the choice of local isomorphism $W_\bullet M\cong W_\bullet\V_p$, to $tg\V_p$ for $t\in P(\Z_p)/(P(\Z_p)\cap gKg^{-1})$. We now show this data is equivalent that of the statement of (2).  An element $t\in P(\Z_p)$ preserves $g\V_p$ precisely if $t\in P(\Z_p)\cap gKg^{-1}$ and it preserves $\Gr_W g\V_p$ precisely if $t\in P(\Z_p)\cap (gKg^{-1} U(\Q_p) )$. Since $g$ is semi-positive, it follows from Lemma \ref{lem:semipositivityproperty} that such a $g\V_p$ is determined by its associated graded $\Gr_W g\V_p\subset \Gr_W V$. Thus, given such $M'_i$ we can uniquely recover $M'$, and the claim follows.

\end{proof}

\begin{definition}
  Let $K_1,K_2\subset K$ and $g\in G(\A_f)$ such that $g^{-1}K_1g\subset K_2$. Then we obtain a natural \textit{Hecke morphism} $\pi_{g,K_1,K_2}:S_{K_1}\to S_{K_2}$ given over $\C$ by $$G(\Q)\backslash X\times G(\A_f)/K_1\ni[(x,h)]\to [(x,hg)]\in G(\Q)\backslash X\times G(\A_f)/K_2.$$ If context is clear, we simply write $\pi_g$, omitting the level structures.
\end{definition}

We record the following properties of Hecke morphism that shall be crucial to us:
\begin{lemma}\label{lem:define mu(p)}
There are Hecke morphisms $\pi_{\mu(p^n)}:S_{P}\to S$.
\end{lemma}

We shall require the following lemma:

\begin{lemma}\label{lem:heckelocsyspullback}
    Let $\pi_g:S_{K_1}\to S_{K_2}$ be a Hecke morphism. Then on $S_{K_1}$, the the pair $(\V_{\et},\pi_g^*\V_{\et})$ has a $\big(K_1, (\V_p,g\V_p)\big)$ - structure. 
\end{lemma}

\begin{proof}
The natural isomorphism $V_p\to \pi_g^*V_p$ identifies $\V_p$ with $g\V_p$ over $\C$, and the claim follows.    
\end{proof}

\subsection{$p$-Hecke orbit}

We require a version of Hecke-correspondence that extends to characteristic $p$. Let $\mathscr{S}$ denote the integral model for $S$ at a good prime $p$. For $g\in G(\Q_p)$ we let $T_g\subset S\times S$ denote the corresponding image of the Hecke correspondence, and its closure $\mathscr{T}_g:=\ol{T_g}\subset\mathscr{S}\times \mathscr{S}$. We refer to $\mathscr{T}_g$ as the \textit{integral} $p$-Hecke correspondence.

\begin{definition}\label{def:pheckeorbit}
    Let $X/\Z_p$ be a scheme. We say that two maps $f_1,f_2:X\to \mathscr{S}$ are in the same \textit{$p$-adic Hecke orbit} if $(f_1,f_2):X\to \mathscr{S}\times \mathscr{S}$ factors through $\mathscr{T}_g$ for some $g\in G(\Q_p)$. This is an equivalence relation, as follows from the following Lemma.
\end{definition}

\begin{lemma}
    Let $X/\Z_p$ be a scheme. The property of being in the same $p$-Hecke orbit is an equivalence relation.
\end{lemma}

\begin{proof}
    Reflexivity follows from $g=\id_G$ and symmetry is obvious, so it remains to prove transitivity. So suppose $f_1,f_2,f_3:X\to \mathscr{S}$ are morphisms with $(f_1,f_2):X\to \mathscr{S}\times \mathscr{S}$ factoring through $\mathscr{T}_g$ and $(f_2,f_3):X\to \mathscr{S}\times \mathscr{S}$ factoring through $\mathscr{T}_h$. It follows that $(f_1,f_3):X\to \mathscr{S}\times \mathscr{S}$ factors through the image of $\mathscr{T}_{g_1}\times_{\mathscr{S}}\mathscr{T}_{g_2}$ in $S\times S$. Now in characteristic $0$, the image of $T_{g_1}\times_S T_{g_2}$ in $S\times S$ factors through $\cup T_h$ for some finite union of $h\in G(\Q_p)$ spanning  left coset representatives for the double coset $K_pg_1K_pg_2K_p$. Taking Zariski-closure completes the proof.
\end{proof}

\section{Background: Crystals}

\subsection{Lattices}
The main goal of this section is to prove the following lemmas. 

\begin{lemma}\label{lem:crystaluniquelyextendstobdrypoints}
    Let $C/k$ be a smooth irreducible curve, and $\ol{C}$ its proper smooth compactification. Let $\E$ be an $F$-crystal on $C$. Then $\E$ extends to a log-F-crystal on $\ol{C}$ it at most one way.
    
\end{lemma}
\begin{proof}
    Suppose $\E',\E''$ are two log-F-crystals on $\ol{C}$ which extend $\E$. 
Kedlaya \cite[Thm 5.3 + Thm 7.3]{kedlaya2022notes} proved the fully-faithfullness of the restriction functor on the level of isocrystals. We thus have a canonical isomorphism $\iota:\E'[\frac1p]\to \E''[\frac1p]$ compatible with the identity morphism on $\E$. We must show that $\iota$ is induced from an isomorphism between $\E'$ and $\E''$.

To that end, Let $D=\ol{C}\backslash C$, and let $(\ol{\cC},\cD)$ be a log-smooth lift to $W(k)$ of $(\ol{C},D)$. Then evaluating $\E',\E''$ on the lift $\ol{\cC},\cD$ give vector bundles $\cE',\cE''$ with a log-connection on $\ol{\cC}$, and $\iota$ gives an isomorphism on the generic fibers $\iota_\eta: \cE'_\eta\to \cE''_\eta$. By assumption on $\iota$, we know that $\iota_\eta$ extends to an isomorphism over $\ol{\cC}\backslash D$. Since morphisms of vector bundles extend over codimension 2, we see that $\iota_\eta$ extends to an isomorphism of vector bundles over all of $\ol{\cC}$. Since it also must respect the connection, the proof is complete.

\end{proof}

We will also need the following result. 
\begin{lemma}\label{lem:crystalsagreeingatallclosedpoints}
    Let $C/k$ be a smooth curve and let $\E_1, \E_2 \subset E$ be two $F$-crystals inside an $F$-isocrystal $E$. Suppose that for every closed point $x\in C(\bar{k})$, we have an equality of $F$-crystals $\E_{1,x} = \E_{2,x}$ inside the isocrystal $E_x$. Then, we have $\E_1 = \E_2$.
\end{lemma}
\begin{proof}
    By the Lemma \ref{lem:crystaluniquelyextendstobdrypoints}, it suffices to replace $C$ by a Zariski-open. Therefore, we may assume that $C$ is affine, and there is a smooth lift $\cC/W(k)$ such that the vector bundles underlying $\E_1(\cC)$ and $\E_2(\cC)$ are free. We are now reduced to proving the following statement. Let $R = \cO(\cC)$\footnote{$R$ is a smooth and $p$-adically complete $W=W(\overline{k})$-algebra}, let $M$ be a finitely generated free $R[1/p]$-module, and let $M_1,M_2 \subset M$, be free $R$-lattices. Suppose that the restriction of $M_1$ and $M_2$ to every $W$-point of $\cC$ agree. Then, it suffices to prove that $M_1 = M_2$. 
    
    To that end, let $e_1 \hdots e_n$ (resp. $f_1 \hdots f_n$) be a basis for $M_1$ (resp. $M_2$).  Let $n$ denote the smallest non-negative integer such that $p^nM_1 \subset M_2$. Let $A \in M_n(R)$ be the change-of-basis matrix that expresses the $p^n e_i$ in terms of the $f_i$. The hypothesis on $n$ implies that there exists at least at least one entry $g = g_{i,j}$ of $A$ which is not a multiple of $p$. By permuting the bases $e_i$ and $f_i$, we may assume that $i = j = 1$. As $g\notin  pR$, we have that there is a $g \not \equiv 0 \mod p$, and so there is a point on $x\in C$ such that $g(x) \neq 0$. Therefore, any $W$-point $\tilde{x}$ of $\cC$ lifting $x\in C$ satisfies $g(x) \not \equiv 0 \mod p$. It follows that $p^n e_1|_{\tilde{x}}$ is a primitive $W$-linear combination of the $f_i|_{\tilde{x}}$. Therefore, the only way for $e_1|_{\tilde{x}}$ to be a $W$-linear combination of the $f_i|_{\tilde{x}}$ is for $n$ to equal 0. It follows that $M_1\subset M_2$. Repeating this argument with the roles of $M_1$ and $M_2$ reversed yields the lemma.

\end{proof}

\begin{lemma}\label{lem: ordinary crystals agreeing at a point}
Let $C$ be a connected smooth curve, and let $\E_1,\E_2 \subset E$ be two $F$-crystals inside an $F$-isocrystal $E$. Suppose that $E$ is ordinary, and that $\E_1$ and $\E_2$ agree at a point. Then, $\E_1 = \E_2$.
\end{lemma}
\begin{proof}
    The ordinariness of $E$ implies that the slope filtration on $E$ induces a slope filtration on $\E_1,\E_2$. We will first prove that the associated graded crystals $\Gr^\bullet \E_1 = \Gr^\bullet \E_2$ agree inside $\Gr^\bullet E$. It suffices to show that $\Gr^{i}\E_1 (-i) = \Gr^i\E_2 (-i)$. Each of these crystals is a unit-root $F$-crystal and so is equivalent to a $\Z_p$-local system on $C$ \cite{crew85}. Therefore, the equality follows by the fact that they agree at a point. 

    We therefore have that $\Gr^\bullet \E_1$ and $\Gr^{\bullet} \E_2$ are the same. Let $x\in C$ be some closed point, and consider $\E_{1,x},\E_{2,x}\subset E_x$. The slope filtration admits a canonical splitting at the level of $F$-isocrystals, and therefore at the level of crystals. The equality of the $F$-crystals at $x$ follows from the equality of the graded crystals at $x$. As $x$ was arbitrary, the lemma follows from Lemma \ref{lem:crystalsagreeingatallclosedpoints}.

\end{proof}

\subsection{Mod $p$ Hodge filtrations and $F$-crystals}

Let $X/k$ be a smooth variety and let $\D/X$ be an $F$-crystal on $X$. We define a filtration by coherent subsheaves on the vector bundle $\D(X)$ as follows. Zariski-locally, $X,F$ admit lifts $\tilde{X},\tilde{F}$ to $W$. Replacing $X$ by one such open, the $F$-structure on $\D$ induces a map $\frob_{\tilde{F}}: \tilde{F}^*(\D(\tilde X)) \rightarrow \D(\tilde{X})$. Consider the coherent sub-module $M_i$ defined by sections $v\in \tilde{F}^*(\D(\tilde{X}))$ such that $p^i \mid \frob_{\tilde{X}}(v)$. Define $\Fil_{F}^i \subset F^*(\D(X))$ to equal the image of $M_i$ inside $F^*(\D(X))$ (note that $F^*(\D(X))$ is just $ \tilde{F}^*(\D(\tilde{X})) \bmod p$. Now, define $\Fil^i \subset \D(X)$ to equal $\Fil_F^i \cap (\D(X)\otimes 1) \subset F^*(\D(X))$. 
\begin{definition}\label{def: mod p hodge filtration in terms of F}
    Define the filtration $\Fil^i \subset \D(X)$ to be the Hodge filtration on $\D$ mod $p$.
\end{definition}
At each point $x$, $\Fil^i$ is just obtained by considering $\D_x$ as a free $W(k_x)$ module with a semilinear endomorphism $\frob_x$, and considering the image inside $\D_x \bmod p$ of the submodule consisting of vectors $v$ that satisfy $p^i \mid \frob_x(v)$. 

\begin{definition}\label{def: acceptable crystal}
    Notation as above. We say that $\D$ is acceptable if the filtration defined in Defition \ref{def: mod p hodge filtration in terms of F} is a filtration by sub-bundles. 
\end{definition}

Now, consider the special setting where $\D/X$ is obtained by starting with a smooth lift $\tilde{X}/W$ and the data of a Fontaine-Laffaille module $(\V,\nabla,\Fil^\bullet,\frob)$. Let $\cris\V/X$ denote the $F$-crystal on $X$ induced by this Fontaine-Laffaille module, and let $\mathcal{V} := {\cris\V}(X)$. Note that this is canonically identified with the vector bundle $\V \bmod p$ on $X$. Then, the Hodge filtration on $\cris\V$ mod $p$ agrees with the filtration $\Fil^\bullet$. Therefore, an F-crystal induced by a Fontaine-Laffaille module is always acceptable.

We now specialize to the setting of Shimura varieties. Let $C/k$ be a smooth curve, $\bar{C}$ its smooth compactification, and let $f: C \rightarrow \integralShimK$, and let $f^{\tor}: \bar C \rightarrow \integralShimK^{\tor}$ be the induced maps. 

\begin{definition}\label{def: Griffiths degree}
    Define the \textit{Griffiths degree} of $(C,f)$ to be the degree of the pullback of the Griffiths bundle of $\cris\V'$ to $\bar{C}$ by $f^{\tor}$. 

\end{definition}
\begin{remark}
    In the setup above, let $f^{\bb}$ denote the map $\bar{C}\rightarrow \cS^{\bb}$. Then the Griffiths degree of $(C,f)$ is of course just the degree of the Grifiths bundle on $\cS^{\bb}$ pulled back to $\bar{C}$ by $f^{\bb}$. 
\end{remark}

\begin{proposition}\label{prop: Griffiths Degree determined by F-crystal}
    The Griffiths degree of $(C,f)$ depends only on the isomorphism class of $f^* (\cris\V')$. In other words, suppose that $f_1,f_2: C\rightarrow \integralShimK$ are two maps such that $f_1^*(\cris\V') \simeq f_2^*(\cris\V')$. Then, the Griffiths degree of $(C,f_1)$ is the same as the Griffiths degree of $(C,f_2)$.
\end{proposition}
\begin{proof}
    By Lemma \ref{lem:crystaluniquelyextendstobdrypoints}, we have an isomorphism of log F-crystals on $C$ induced by the maps $f_1^{\tor}$ and $f_2^{\tor}$ (and this isomorphism extends the isomorphism $f_1^*{\cris\V'} \simeq f_2^*(\cris\V')$). Evaluating these on $C$, we have an isomorphism of bundles with log-connection $f_1^{\tor,*}(\dR\cV') \simeq f_2^{\tor,*}(\dR\cV')$. 

    We further have that the Hodge filtration on $f_i^*(\cris\V') \bmod p$ (as defined in Definition \ref{def: mod p hodge filtration in terms of F}) agrees with the filtration $f_i^*(\Fil^{\bullet}({\dR}\cV'))$, as $\Fil^\bullet({\dR}\cV')$ is the filtration induced by the Fontaine-Laffaille structure on $\cV$, which agrees with the mod $p$ Hodge filtration on $\cris\V'$. Therefore, the isomorphism $f_1^{\tor,*}(\dR\cV')|_C \simeq f_2^{\tor,*}(\dR\cV')|_C$, which is induced by an isomorphism of log $F$-crystals, sends $f_1^*(\Fil^\bullet ({\dR}\cV'))$ isomorphically to $f_1^*(\Fil^\bullet ({\dR}\cV'))$. It follows therefore that $f_1^{\tor,*}(\Fil^{\bullet} ({\dR}\cV'))$ maps isomorphically to $f_2^{\tor,*}(\Fil^{\bullet} ({\dR}\cV'))$, and therefore the Griffiths degrees are the same. 
\end{proof}

\section{Ordinary Locus of Shimura Varieties}

\subsection{The ordinary locus}

We first recall the notion of ordinariness.  Let $T\subset B\subset G$ be a maximal torus that splits over an \'etale extension of $\Z_p$ (that $G/\Z_p$ is reductive implies that these objects exist). Let $x \in \integralShimK(\F)$. Associated to the $F$-crystal $\cris\V'_x$ is the Frobenius element $b_x\in G(L)$, which is well defined up to $\sigma$-conjugation by $G(W)$. Let $\nu_x \in X_*(T)_\Q$ be the dominant element that is conjugate to $\nu_{b_x}$, the Newton cocharacter associated to $b_x$ (see work of Kottwitz \cite{Kott1} and \cite{Kott2} for definitions and more details). Note that the conjugacy class of $\nu_{b_x}$ depends only on the $\sigma$-conjugacy class of $b_x$, and so the definition of $\nu_x$ is independent of the choice of representative of the $\sigma$-conjugacy class of $b_x$. We note that $\nu_x$ is defined over $\Z_p$. Let $\mu_x\in X_*(T)$ denote the dominant co-character with $b_x \in G(W) \sigma(\mu_x)(p) G(W)$. We have that $\nu_x \preceq \mu_x$ (for example, see \cite{Gashi}), where two dominant fractional cocharacters satisfy $\mu' \preceq \mu''$ if and only if $\mu'' - \mu'$ can be written as a non-negative rational linear combination of positive co-roots. We have the following definition. 

  \begin{definition}\label{def:ord}
   The $F$-crystal $\cris\V'_x$ - and by extension the point $x$ -  is said to be ordinary if $\nu_x = \mu_x$.

  \end{definition}
\subsection{Filtrations on local systems}

\begin{lemma}\label{thm:etaleslopefiltration}
There exists a descending filtration $W^\bullet{_{\et}\V_p}$ on $_{\et}\V_p|_{\ShimK^{\ord}}$ such that for each $i$ the cyclotomic twist $(\Gr^i_W{_{\et}\V_p}|_{S^\ord})(-i)$ extends to $\integralShimK^{\ord}$, and the filtration is uniquely determined by this property.  Finally, we have that $ W^{\bullet}{_{\et} \V_p}$ is a $(P\subset G)$-filtration of $\V_p|_{S^{\ord}}$.

\end{lemma}

\begin{proof}
For an $F$-crystal $\M$ over a point whose Hodge and Newton polygons are equal with integer slopes, there is a unique filtration (in fact splitting) by $F$-crystals $W^\bullet\M\subset \M$ with the property that $\Gr^i_W\M$ is isoclinic of slope $i$ \cite[Theorem 1.6.1]{katz_crystal}.  These filtrations behave well in families in the following sense.  Let $X$ be a smooth scheme in characteristic $p$ and $\M$ an $F$-crystal whose Hodge and Newton polygons are pointwise constant, equal, and with integer slopes.  Then by a theorem of Katz \cite[Theorem 2.4.2]{katz_crystal}, there is a filtration by sub-$F$-crystals $W^\bullet \M$ whose underlying filtration of coherent sheaves (on a choice of local lifting) is locally split, whose graded subquotients $\Gr^i_W\M$ are isoclinic of slope $i$, and which is compatible with pullback (in particular, with restriction to a point).

If $\M$ underlies a Fontaine--Laffaille module (on a $p$-adic formal scheme $\cX$ with special fiber $X$), we claim this filtration is by sub-Fontaine--Laffaille modules. We claim $W^0\M\cap \Fil^1\M=0$. Indeed, since both sub-bundle are saturated, it is enough to check the statement after reducing mod $p$. But on $X$ itself, the Frobenius operator is invertible on $W^0\M_X$ and vanishes on $\Fil^1\M_X$, proving the claim.  Now the claim is true for the twisted quotient Fontaine-Laffaille module $(\M/W^0\M)(1)$ and therefore also for $\M/W^0\M$ .  Finally, the $W^i\M$ are the full pre-images of $W^i(\M/W^0\M)$, completing the proof.

To finish showing the existence of this filtration, we first work with $_{\et}\V_p':=_{\et}\V_p(m)|_{\ShimK^{\ord}}$  By the above, the $F$-crystal $ {\cris\V'}$ is filtered by $F$-crystals underlying Fontaine--Laffaille modules, which therefore corresponds to a filtration by local systems.  Moreover, since every Fontaine-Laffaille module of weight $0$ extends to a local system over $\integralShimK^{\ord}$, the graded subquotients extend up to a cyclotomic twist. Twisting down, we obtain a filtration $W^\bullet:=W^\bullet {_{\et}\V_p}$ of $_{\et}\V_p$.

We now need to prove that $W^\bullet \subset {_\et\V_p}$ is a $P$-filtration. We will first show this at the level of $\Q_p$ local systems. It suffices to work with a fixed connected component $\cS_0$ of $\cS$. Let $x\in \cS_{0,k}^{\ord}(\bar k)$ be a point and let $\tilde{x}$ be its canonical lift. By \cite{Wortman} (also see \cite[Section 7,8]{BST}) we may choose coordinates on $\cris V_x$ so that $\frob_x = \mu(p)\sigma$ where $\mu$ induces the Hodge filtration on $\cris V_x$ that defines $\tilde{x}$. It follows that the filtered isocrystal associated to $\tilde{x}$ has $P$-structure. Note that weight-spaces of $\mu$ on $\cris V_x$ provide a splitting to the slope filtration, and so the $P$-structure is compatible with the slope filtration. By applying Fontaine's equivalence of categories between crystalline representations and weakly admissible filtered $F$-isocrystals, we obtain a $P$-structure on $_{\et}V_{p,\tilde{x}}$. As Fontaine's equivalence is compatible with the Fontaine-Laffaille correspondence, we see that the filtration induced by the $P$-structure on $_{\et}V_{p,\tilde{x}}$ is the same as $W^{\bullet}_{\tilde{x}}$. As $W^{\bullet}$ is a filtration by sub-local systems, it follows that the monodromy representation is given by a map $\pi_1(S_0^{\ord},\tilde{x}) \rightarrow P(\Q_p)$. It follows that $W^{\bullet}\subset {_{\et}V_{p}}$ is a $P(\Q_p)$-filtration.

In order to conclude, observe that the $G(\Z_p)$-structure on $_{\et}\V_p$ (resp. $P(\Q_p)$-structure on $_{\et}V_p$) gives a $G(\Z_p)$-orbit (resp. $P(\Q_p)$-orbit) inside the set of $G(\Q_p)$-trivializations of $_{\et}V_p$. It suffices to prove that these two orbits intersect. The set of $G(\Q_p)$-trivializations of $_{\et}V_p$ is a principal homogenous space for $G(\Q_p)$, and so the question reduces to proving that a $G(\Z_p)$-coset and $P(\Q_p)$-coset must intersect. This  follows from the existence of the Iwasawa decomposition which states $G(\Z_p)\cdot P(\Q_p) = G(\Q_p)$.   



\end{proof}

\begin{definition}
    Define $\Gr_W^{\et}(_\et\V_p) :=\bigoplus_i\Gr_W^i( _{\et} \V_p|_{S^{\ord}})(-i)$. We have that $\Gr_W^{\et}(_{\et}\V_p)$ extends naturally to a local system on $\cS^{\ord}$ with $L(\Z_p)$-structure, where $L$ is the reductive quotient of $P$, and we consider it as such.
\end{definition}

More generally, for any ordinary $F$-(iso)crystal $\M$ on a scheme $X$, \cite[Theorem 2.4.2]{katz_crystal} gives a filtration by sub-$F$-crystals $T^\bullet \M$ whose underlying filtration of coherent sheaves (on a choice of local lifting) is locally split, whose graded subquotients $\Gr^i_T\M$ are isoclinic of slope $i$, and thus by \cite[Thm 2.1]{crew85},  $\Gr^i_T\M(-i)$ corresponds to an \'etale local system.

\begin{definition}
For an ordinary $F$-(iso)crystal $\M$ on $X$, we define $\Gr_W^{\et}(\M) :=\bigoplus_i\Gr^i_T\M(-i)$ thought of as an \'etale local system. 
\end{definition}

Note that $\Gr_W^{\et} (M(-1))\cong \Gr_W^{\et} M$.

\newcommand{\ordsection}{\lambda}
\subsection{Hecke operators over the ordinary locus}
Putting \Cref{thm:etaleslopefiltration} and \Cref{lem:lift} together we obtain:
\begin{corollary}The natural morphism $S_{P}\to \ShimK$ admits a canonical section $\ordsection: \ShimK^\ord\to S_{P}$ over $\ShimK^\ord$.
\end{corollary}

It follows that for any Hecke morphism $\pi:S_{P}\to S$ we obtain a morphism $\pi(\ordsection):=\pi\circ \ordsection:\ShimK^\ord\to \ShimK^\ord$.  In particular, this applies to the Hecke correspondence $\pi_{\mu(p)}:S_{P}\to \ShimK$ induced by $\mu(p)$ (see \Cref{lem:define mu(p)}), and we define $\tau:=\pi_{\mu(p)}(\ordsection):\ShimK^\ord\to \ShimK^\ord$.

\begin{proposition}\label{prop: canonical branch} The morphism $\tau:\ShimK^\ord\to \ShimK^\ord$ extends to an integral morphism $\tau:\integralShimK^{\ord} \rightarrow \integralShimK^{\ord}$.  Moreover, $\tau$ reduces mod $p$ to Frobenius.  
\end{proposition}

\begin{proof}
The first statement is a consequence of the following two lemmas:

    \begin{lemma}\label{lemma:quasiaffine}
        $\cS_{k}^{\ord}$ is quasi-affine.
    \end{lemma}
    \begin{proof}  First observe that if $X/k$ is a scheme over a characteristic $p$ field, then any line bundle $L$ with a nonzero section $s$ and such that $F_{X/k}^*L\cong L$ where $F_{X/k}:X\to X$ is the relative Frobenius is necessarily torsion.  Indeed, $s^p/F_{X/k}^*s$ is a nowhere zero section of $L^{p-1}$.  As quasi-affineness is equivalent to $\cO_X$ being ample, it follows that $X$ is quasi-affine if and only if it admits an ample line bundle $L$ for which $F_{X/k}^*L\cong L$.

    To finish, recall that we have a mod $p$ Fontaine--Laffaille module $(M,\nabla,\Fil^\bullet M)$ on $\cS_{k}$, and that the conjugate filtration $C^\bullet M$ is an increasing filtration of $(M,\nabla)$ by flat locally split subsheaves satisfying:  (1) $\gr^C_i M\cong F_{\cS_{k}/k}^*\gr^i_\Fil M$ and (2) over $\cS^{\ord}_{k}$, each $C_iM$ is a complement to $\Fil^{i+1}M$.  It follows that, over $\cS^{\ord}_{k}$, we have $F_{\cS_{k}/k}^*(\det \Fil^{i}M)\cong \det \Fil^{i}M$, and therefore also $F_{\cS_{k}/k}^*L\cong L$ where $L=\otimes_i \det \Fil^iM $ is the Griffiths bundle.  As $L$ is ample, the conclusion follows from the previous paragraph. 
    \end{proof}
    
\begin{lemma}\label{lem:quasiaffineextension}
Let $\mathscr{X},\mathscr{Y}/\cO_{E_v}$ be two formal $p$-adic schemes.  Assume $\mathscr{X}$ is integral and normal and that the special fiber $\mathscr{Y}_k$ of $\mathscr{Y}$ is quasi-affine.  Then any morphism $f:\mathscr{X}^{\mathrm{rig}}\to\mathscr{Y}^{\mathrm{rig}}$ of rigid generic fibers extends to a morphism $\bar f:\mathscr{X}\to\mathscr{Y}$ of formal models.
\end{lemma}
\begin{proof}
 According to Raynaud \cite[\S8.4,Thm 3]{bosch2005}, the rigid analytic morphism $f:\mathscr{X}^{\mathrm{rig}}\to\mathscr{Y}^{\mathrm{rig}}$ extends to a morphism of formal models $\tilde f:\mathscr{X}'\to\mathscr{Y}$ for some admissible blow-up $g:\mathscr{X}'\to \mathscr{X}$.  The morphism $g$ is proper, and therefore by the quasi-affineness of $\mathscr{Y}_k$ the fibers of $g$ are contracted under $\tilde f$.  By the rigidity lemma \cite[Lemma 1.15]{debarre} (see also \cite[Lemma 2.17]{BST} - the proof is the same in the category of formal schemes), the morphism $\tilde f$ factors through a morphism $\bar f:\mathscr{X}\to \mathscr{Y}$. 
\end{proof}

For the second claim, it suffices to prove that $\tau \bmod p$ agrees with Frobenius for a dense set of points. By \cite[Thm 8.2]{BST}, the ordinary locus is open dense and that every ordinary $x\in S(\Fpbar)$ admits a canonical lift. We will do this for a sufficiently large class of ordinary points -- to that end, let $(T,h)$ be a 0-dimensional Shimura variety, with $h$ defined over a reflex field $E_T$. Let $\fp$ be a height-1 prime of $E_T$. Let $S:=S_{T(\hat{\Z})}(T,h)$, equipped with its integral canonical model $\cS$ at $E_{\fp}$. We have the lemma below.
\begin{lemma}\label{lem: reciprocity special points}
Consider the Hecke morphism $\pi_{h(p)}:\cS\to \cS$ induced by $h(p)$.  Then $\pi_{h(p)}$ reduces to Frobenius mod $\fp$.

\end{lemma}

\begin{proof}
We use the Reciprocity law. The map $h$ induces a map $\hat{h}: \A^{\times}_{E,f}\to T(\A_{E,f})$. Composing this with the trace map from $E$ to $\Q$ gives a map $r:\A^{\times}_{E,f}\to T(\A_f)$.

Let $s\in \A^{\times}_{E,f}$ be such that $s_{\fp}=p$ and $s$ has component $1$ at all other places. Let $\phi=\mathrm{art}_E(s)$. Finally, take $x\in \cS(\cO_{E,\fp}^{\mathrm{un}})$. By \cite[(62)]{MilneIntro} it follows that
$\pi_{h(p)}(x)=\phi(x)$. Reducing modulo $\fp$ yields the result, since $\phi$ reduces to Frobenius.

\end{proof}

Let $x\in \integralShimK^{\ord}(\Fpbar)$. The point $x$ has a canonical lift $\tilde{x}$, which by \cite[Thm 8.7]{BST} is special. This gives us a zero-dimensional Shimura datum $(T,h)$, an embedding $(T,h)\rightarrow (G,X)$, and a choice of prime $\fp $ of the reflex field $E_T$ of $(T,h)$ that lies over our chosen prime of the reflex field of $(G,X)$. It suffices to prove that the set of points $x$ such that $\fp$ is a height-1 prime of $E_T$ is dense in $\cS_k^{\ord}(\Fpbar)$. We will do this in two steps. Given a point $x\in \cS^{\ord}(\Fpbar)$, we let $\cris\frob_x$ denote the crystalline Frobenius on $\cris\V_x$, and we let $_{\et}\frob_{x}$ to be the endomorphism of $\Gr^\et_WV_{p,\bar{x}}$ induced by the $q$-Frobenius in $\textrm{Gal}(\Fpbar/\F_q)$, where $\F_q$ is the field of definition of $x$. We may consider $_{\et}\frob_{x}$ as an element of $L(\Z_p)$ upto conjugacy. We will first show that there is a Zariski-dense set of points $x\in \cS_k^{\ord}(\Fpbar)$, such that the centralizer of $_{\et}\frob_x$ in $L$ is an unramified torus. We will then prove that for such points $x$, the associated Shimura data $(T,h)$ and the prime $\fp$ of the reflex field $E_T$ that induces the canonical lift of $x$ is height 1. 



We now carry out the first step. First note that the $\ell$-adic local systems and the $F$-isocrystals are companions by \cite[Thm 1.1]{patrikis2025}. As every open subvariety of $\integralShimK^{\ord} \mod p$ contains a curve with big monodromy, it suffices to prove that such a curve that also intersects the ordinary locus has infinitely many such points. Let $C \subset \integralShimK^{\ord}$ be such a curve. As $\cris V$ and $_{\et}V_\ell$ are companions, we have that $\cris V|_C$ has overconvergent monodromy $G$, and by the Parabolicity conjecture proved by \cite{Marco}, we have that $\cris V|_C$ has convergent monodromy $P\subset G$. Finally, consider the unit-root crystal $\Gr^{\bullet}_{\et}(\cris\V)$ - this is now has monodromy contained in $L$, but which contains $L^{\der}$. The monodromy $L_C$ of the actual \'etale local system asociated to this unit root $F$-crystal will therefore contain an open subgroup of $L^{\der}(\Z_p)$. There is an open (in the $p$-adic topology) subset of $L_C$ that consists of elements whose centralizer in $L$ is an unramified quasi-split\footnote{Recall that we define maximal torus in a reductive group to be quasi-split if it contains a maximally split torus.} maximal torus. Indeed, consider an unramified torus $Y\subset  L^{\der}$, and set $Y_r\subset Y(\Z_p)$ to be the set of regular elements. Then, the set $L^{\der}_{C,r}$ of elements in $Y_C \cap L^{\der}(\Z_p)$ conjugate to some element of $Y_r$ is open in $L_C \cap L^{\der}(\Z_p)$ (see for eg \cite[Claim 2.18]{lamshankar2025}). The open subset we require is just $L_{c,r} = L^{\der}_{C,r} \cdot Z(L_C)$, where $Z(L_C) \subset L_C$ is defined to be the intersection of the center of $L$ with $L_C$. Applying the Chebotarev density theorem now shows that there are infinitely many points $x\in C$ such that the (conjugacy class) of $\varphi_x$ is contained in $L_{C,r}$. 

We will now show that for such a point $x$ and associated Shimura datum $(T,h)/E_T$, the embedding of Shimura data is induced by a height-1 prime of $E_T$. By solely working in the setting of $S(T,h)$ and the prime $\fp \subset \cO_{E_T}$, we have that $_{\et}\frob_x$ is induced by an element of $T(\Q_p)$. On the other hand, by treating $x$ as a point of $\integralShimK$, we have that the centralizer in $L$ of any element in the conjugacy class defined by $_{\et}\frob_x$ is an unramified torus. Therefore, we have that $T$ is contained in an unramified torus of $L^{\der}$ up to conjugacy, and therefore itself is unramified. 

Embedding $E_T$ in $\overline{\Q}_p$ via $\fp$, we have the co-character $h: \mathbb{G}_{m,{E_{T,\fp}}} \rightarrow T_{E_{T,\fp}}$ defined over $E_{T,\fp}$. The fact that $T$ is unramified implies that $E_{T,\fp}$ is an unramified extension of $\Q_p$. 

By \cite[Thm 8.4]{BST} and the fact that $x$ is an ordinary point of $\cS$, the Hodge and slope filtrations agree under the identification ${\cris}V_x \simeq _{\dR}V_{\tilde{x}}$. We now think of $\cris V$ as an $F$-isocrystal with $T$-structure. Since $T$ is commutative, we have a canonical action of $T$ on $\cris V$.
By \cite{Kott1}, the Newton co-character $\nu: \Gm \rightarrow T$ is defined over $\Q_p$, and induces the slope filtration. Since $h$ induces the Hodge filtration, it follows that $h$ is defined over $\Q_p$, whence $E_{T_{\fp}} = \Q_p$ as required.

\end{proof}

\begin{remark}
    Ayan Nath, Keerthi Madapusi's student, in ongoing work also proves that Frobenius on $\cS^{\ord}_k$ admits a lift to the ordinary locus using work of Gardner-Madapusi (\cite{GardnerMadapusi}).
\end{remark}

\subsection{Integral lifting lemma}
We will need an integral version of \Cref{lem:lift}, but first we will need to know the \'etale cover $\pi_U:S_U\to S_P$ extends integrally, at least over the image of the canonical section $\ordsection$.  For any finite cover $\pi:S_{K'}\to S_P$ let $S^\ordsection_{K'}:=S^\ord\times_{S_P}S_{K'}$ be the pullback along $\ordsection$.  Note that $S^\ord\cong S^\ordsection_P$ is just the image of the section, and therefore has an integral model $\mathscr{S}^\ord=:\mathscr{S}^\ordsection_P$.

\begin{lemma}\label{lem:intstructurecover}
    \begin{enumerate}
        \item For any semi-positive $g$, the natural morphism $\pi:S^\ordsection_{P(\Z_p)K^p\cap gKg^{-1}}\to S^\ordsection_P $ extends naturally to a finite \'etale cover $\mathscr{S}^\ordsection_{P(\Z_p)K^p\cap gKg^{-1}}\to \mathscr{S}^\ordsection_P$. 
        \item The morphism $\pi_g:S^\ordsection_{P(\Z_p)K^p\cap gKg^{-1}}\to S^{\ord}$ extends integrally to $\pi_g:\mathscr{S}^\ordsection_{P(\Z_p)K^p\cap gKg^{-1}}\to \mathscr{S}^{\ord}$.
        \item We obtain a natural isomorphism $\pi^*{_{\et}V_{p}}\to \pi_g^* {_{\et}V_{p}}$, such that $(\pi^*_{\et}V_{p},{\pi_g^*} _{\et}V_{p})$ has a natural $\big(P(\Z_p)\cap gK_pgK^{-1}, \V_p,g\V_p\big)$ structure.
        \end{enumerate}
\end{lemma}

\begin{proof}
We first consider the finite \'etale cover $S_P^g$, consisting of, on $S^{\ordsection}_P$,  the space of graded sub-local systems $\oplus_i A_i\subset \Gr_W {_{\et}\V'_p}\otimes \Q_p$, for which $(\Gr_W{_{\et}}\V_p,\oplus_i A_i)$ has a $\big(L(\Z_p),(\Gr_W \V_p, \Gr_W g\V_p)\big)$-structure compatible with the $P(\Z_p)$-structure on $(\V_p,W_\bullet\V_p)$. By Lemma \ref{lem:lift} there is a natural isomorphism $S_P^g\to S^{\ordsection}_{P(\Z_p)K^p\cap gKg^{-1}}$.

Now we wish to use this to define our integral model. Note that while $\Gr_W {_{\et}\V'_p}$ does not extend to $\mathscr{S}^\ord$, the twisted $\Gr^{\et}_W {_{\et}\V'_p}$ does. Moreover, these two local systems on the generic fiber are twists of each other by $\mu$ viewed as an $L$ co-character. Hence, $S^P_g$ is naturally the space of graded sub-local systems $\oplus_i A_i\subset \Gr^{\et}_W {_{\et}\V'_p}\otimes \Q_p$, for which $(\Gr^{\et}_W{_{\et}}\V_p,\oplus_i A_i)$ has a $\big(L(\Z_p),(\Gr^{\et}_W \V_p, \Gr^{\et}_W g\V_p)\big)$-structure compatible with the $P(\Z_p)$-structure on $(\V_p,W_\bullet\V_p)$. This latter definition extends naturally to give our integral structure. This prove (i).

For (ii), the statement follows from \Cref{lemma:quasiaffine} and \Cref{lem:quasiaffineextension}.

Finally, for (iii), the statement is clear on the characteristic $0$ fiber, and extends to mixed characteristic by normality.
\end{proof}

In the following, we denote the extensions of the local systems $(\Gr^{\et}_W {{_{et}\V_p}}|_{S_P^\ordsection})$ to $\mathscr{S}^\ordsection_P$ by $M$.

\begin{theorem}\label{thm:integral lift}
Let $Y/k$ be a variety mod p with a morphism $\phi:Y\to \mathscr{S}^\ord$, $\phi^\ordsection:Y\to \mathscr{S}^\ordsection_P$ its canonical lift.  Let $g\in P(\Q_p)$ be a semi-positive element. A lift $\tilde{\phi^\ordsection}:Y\to \mathscr{S}^\ordsection_{P(\Z_p)K^p\cap gKg^{-1}}$ is equivalent to a graded lattice $M'\subset M_Y\otimes \Q_p$, for which $(M_Y,M')$ has a $(\Gr^{\et}_W{\V_p}, \Gr_W^{\et} g\V_p)$-structure which is compatible with the $L$-structure. Given such a lift, the composition $\psi:=\pi_g\circ\tilde{\phi^\ordsection}:C\to\cS^{\ord}$ satisfies $\psi^*\Gr_W^{\et}{_{\et}\V'_{p}}\cong M'$.
\end{theorem}

\begin{proof}
This follow directly from the construction of $\mathscr{S}^\ordsection_P$ given in Lemma \ref{lem:intstructurecover}.
\end{proof}

\section{Main Argument}

In this section we finally prove \Cref{thm:mainshimuravarieties}:

\begin{theorem}\label{thm:Heckeflattness}
    Let $k$ be  a finite field of characteristic not dividing $N$. Let  $C$ be a smooth irreducible curve over either $k$ or $\bar{k}$. 
    There exists a constant $N$ such that the following holds: 
    Given a generically ordinary map $f: C\rightarrow \integralShimK$, there is a map $f_1: C\rightarrow \integralShimK$ in the $p$-Hecke orbit of $f$ of bounded Griffiths degree. In particular, if $C$ is defined over $k$, there are only finitely many maps from $C$ to $\integralShimK$ up to $p$-Hecke orbit.

    Moreover, if we assume that $C/\ol{k}$ and the map $f$ has maximal monodromy, then the set of $p$-Hecke orbits is finite.
\end{theorem}

\begin{proof}

We first claim that the set of $F$-isocrystals with $G$-structure $f^*{{\cris}V}$ is finite. The claim follows from  
\Cref{thm:Fisocrystalsfinitestuff} (2)+(4)  (resp. (3)+(4)) once we establish the assumptions of (2) (resp. (3)). The semi-simplicity, tameness, and overconvergence follow from those same results for $\cris V$ established in \cite[Thm 1.1(3)+Lem 3.1]{patrikis2025}. For (2), note that the set of slopes of $f^*{{\cris}V}$ are contained in the set of slopes of ${\cris}V$. Moreover, $f^*{{\cris}V}$ are all $\Q_\ell$-nice since $\cris V$ is $\Q_\ell$-nice by \cite[Thm 1.1(3)]{patrikis2025}.
For (3), fix $x\in C(k')$ and note that $f(x)\in \cS(k')$, which is a finite set, and thus there are finitely many possible Frobenius eigenvalues.

Next, by the finiteness above, there exists a dense open curve $C^\circ\subset C$ where each such generically ordinary $F$-isocrystal is ordinary.  It is then sufficient to prove finiteness of $p$-Hecke orbits of maps $f:C^{\circ}\to S^{\ord}$ inducing a fixed $F$-isocrystal with $G$-structure $f^*{\cris V}$. We adopt this perspective, and thereby reduce to $C^{\circ}=C$. Let $c\in C$ be a fixed point. 

Now we may fix an $F$-crystal $\D\subset f^*{\cris V'}$ such that $\Gr_W^{\et} \D$ has an $\big(L(\Z_p),\Gr^{\et}_W\V\big)$-structure compatible with the $L(\Q_p)$-structure on $\Gr_W^{\et}f^*{\cris V'}\cong\Gr_W^{\et}f^*{\cris V}$, with bounded Griffiths degree.

\begin{lemma}
For large enough $m$, there is a semi-positive element $g\in L(\Q_p)$ such that  $(\Gr_W^{\et}f^*{\cris\V'},\Gr_W^{\et}F^{m*}\D)_c$ has an  $\big(L(\Z_p)\cap g^{-1}L(\Z_p)g,(\Gr_W\V,\Gr_Wg\V))$-structure compatible with the ambient $L(\Q_p)$ structure of $(\Gr_W^\et f^*{\cris V'})_c$.
\end{lemma}

\begin{proof}
For any trivialization, there is some element $g\in L(\Q_p)$ such that $(\Gr^{\et}_Wf^*{\cris\V'},\Gr^{\et}_W\D)_c$ acquires a compatible  $\big(L(\Z_p)\cap g^{-1}L(\Z_p)g,(\Gr_W\V,\Gr_Wg\V)\big)$ structure. Thus, $(\Gr^{\et}_W f^*{\cris\V'},\Gr^{\et}_W F^{m*}D)_c$ acquires a $\big(L(\Z_p)\cap g^{-1}L(\Z_p)g,(\Gr_W\V,\Gr_W\mu(p^n)g\V)\big)$ structure. For large enough $n$ the element $\mu(p^n)g$ is semi-positive (by \cref{lem:semipositivityproperty}), completing the proof.
\end{proof}

We let $m\in \N,g\in G(\Q_p)$ denote the corresponding elements in the lemma. Note that we have a map $C\to \mathscr{S}_P^\ordsection$ by composing with $\ordsection:\mathscr{S}\to \mathscr{S}_P$. Now by \Cref{thm:integral lift} we obtain a lift $C\to \cS^{\ordsection}_{P(\Z_p)K^p\cap gK_pg^{-1}}$ corresponding to $\Gr_W^\et F^{m*}\D$, and composing with $\pi_g$ we obtain a morphism $\psi:C\to S$ such that ${\psi}^{*}{\cris\V'}\cong F^{m*}\D$ by \Cref{lem:heckelocsyspullback,lem: ordinary crystals agreeing at a point}. It remains to prove that $\psi$ factors through the $m$'th power of Frobenius. Hence it suffices to show that if the pull-back crystal under $\psi$ is a Frobenius pull-back, then the morphism factors through Frobenius. Now if the $F$-crystal is a Frobenius pullback, the corresponding Kodaira-Spencer map on the flat bundle mod $p$ is trivial.  Indeed, the underlying flat vector bundle is a Frobenius pullback $F^*M$ equipped with its canonical connection for which $1\otimes M$ is flat, and as the Hodge filtration is also pulled back under Frobenius it is flat.  By the immersiveness of the Kodaira-Spencer map (conditions \ref{cond:integralmodel}), it follows that $\psi$ has zero derivative and therefore factors through Frobenius. Therefore,  ${F^{-m}}^*\psi$ is in the same $p$-Hecke orbit as $f$ and has bounded Griffiths degree, as desired. The finiteness of such maps over $k$ now follows from \Cref{prop: Griffiths Degree determined by F-crystal} and the ampleness of the Griffiths bundle assumed in \Cref{cond:integralmodel}.

Finally, we prove the last part of the theorem. Suppose that $C$ is defined over $\ol{k}$ and $f$ has maximal monodromy. By the above, it is sufficient to prove that if $g:C\times X\to S$, with $X$ a smooth and connected variety over $\ol{k}$,  such that $g_{C\times x_0}=f$, then all fibers $g_{C\times x}$ are constant\footnote{This is only because we are in the generically ordinary setting, in general the conclusion would only be that they lie in the same $p$-Hecke orbit}.

To see this, note that image of $\pi^{\et}_1(X)$ in $\ell$-adic monodromy commutes with that of $\pi_1^{\et}(C)$, and is therefore trivial. Hence $g^*{_{\et}V_\ell}$ is pulled back from $C$, and thus the $F$-isocrystal $g^* {\cris V\mid_{X\times c}}\otimes \ol{\Q}_p$ is constant - as an $\ol{\Q_p}$-$F$-isocrystal with $G$-structure - for any $c\in C(\ol k)$. Weil-restricting from a finite extension $L/\Q_p$, we see that $g^* {\cris V\mid_{X\times c}}$ is a direct-summand of a constant $F$-isocrystal, and is therefore constant itself.

Since our $F$-isocrystal is ordinary,  the $F$-crystal is also constant by \Cref{lem: ordinary crystals agreeing at a point}. It follows that the Griffiths bundle is trivial, hence since the Griffiths bundle is also ample on $\cS^{\bb}$ by conditions \ref{cond:integralmodel}, the map $g\mid_{X\times c}$ is constant. Hence $g$ factors through $C$ and the claim follows.

\end{proof}

\bibliographystyle{abbrv} \bibliography{biblio}

\end{document}